\documentclass[12pt]{amsart}

\usepackage{etex}
\usepackage[usenames,dvipsnames]{pstricks} 
\usepackage{epsfig}
\usepackage{graphicx,color}
\usepackage{geometry}
\usepackage[all]{xy}
\usepackage{amssymb}
\usepackage{cite}
\usepackage{fullpage}
\xyoption{poly}

\usepackage{tikz}
\usepackage{tikz-cd}
\usetikzlibrary{decorations.pathmorphing}

\newtheorem*{introtheorem*}{Theorem}

\newtheorem{theorem}{Theorem}[section]
\newtheorem{lemma}[theorem]{Lemma}
\newtheorem{proposition}[theorem]{Proposition}
\newtheorem{corollary}[theorem]{Corollary}
\theoremstyle{definition}

\newtheorem{remark}[theorem]{Remark}

\newtheorem{example}[theorem]{Example}

\newtheorem*{question*}{Question}

\newtheorem*{steps*}{Answer/steps}

\newtheorem*{progress*}{Progress}

\newtheorem*{classification*}{Classification}

\newtheorem*{construction*}{Classification}
\newtheorem*{example*}{Example}

\newtheorem*{remark*}{Remark}
\newtheorem*{remarks*}{Remarks}
\newtheorem*{definition*}{Definition}

\usepackage{calrsfs}
\usepackage{url}
\usepackage{longtable}
\usepackage[OT2, T1]{fontenc}
\usepackage{textcomp}
\usepackage{times}
\usepackage[scaled=0.92]{helvet}

\newcommand{\mc}[1]{\mathcal #1}
\newcommand{\wh}[1]{\widehat{#1}}

\newcommand{\Q}{\mathbb{Q}}

\newcommand{\R}{\mathbb{R}}

\newcommand{\PP}{\mathbb{P}}
\newcommand{\cI}{\mathcal{I}}
\newcommand{\cZ}{\mathcal{Z}}

\newcommand{\cO}{\mathcal{O}}
\newcommand{\cC}{\mathcal{C}}

\newcommand{\X}{\mathcal{X}}

\DeclareMathOperator{\Proj}{Proj}
\DeclareSymbolFont{cyrletters}{OT2}{wncyr}{m}{n}
\DeclareMathSymbol{\Sha}{\mathalpha}{cyrletters}{"58}

\newcommand{\Spec}{\mathrm{Spec}}

\newcommand{\eua}{\mathfrak{a}}

\newcommand{\eum}{\mathfrak{m}}

\newcommand{\ha}[1]{{\hbox to#1pt{}}}
\newcommand{\hb}[1]{{\hbox to-#1pt{}}}
\newcommand{\ilim}[1]{\hbox to14pt{%
{\rm lim}\kern-15pt\lower4.5pt\hbox{$\scriptstyle\longrightarrow\,$}%
\kern-9pt\lower7.5pt\hbox{$\scriptscriptstyle{#1}\,$}}\ha3}
\newcommand{\plim}[1]{\hbox to14pt{%
{\rm lim}\kern-15pt\lower4.5pt\hbox{$\scriptstyle\longleftarrow\,$}%
\kern-9pt\lower7.5pt\hbox{$\scriptscriptstyle{#1}\,$}}\ha3}

\begin{document}

\thanks{The first and second author were supported on NSF grants DMS-1463733 and  DMS-1805439; additional support was provided by a Simons Fellowship (JH). The third author was supported by an Oberwolfach Leibniz Fellowship. The fourth author was supported by NSF collaborative FRG grant DMS-1265290.
	}

\date{\today}
\title{A comparison between obstructions to local-global principles over semi-global fields}
\author{David Harbater, Julia Hartmann, Valentijn Karemaker, and Florian Pop}
\maketitle
\setcounter{tocdepth}{2}

\begin{abstract} 
We consider local-global principles for rational points on varieties, in particular torsors, over one-variable function fields over complete discretely valued fields. There are several notions of such principles, arising either from the valuation theory of the function field, or from the geometry of a regular model of the function field. Our results compare the corresponding obstructions, proving in particular that a local-global principle with respect to valuations implies a local-global principle with respect to a sufficiently fine regular model.
\end{abstract}

\section*{Introduction}
Classical local-global principles originated in the study of algebraic structures over the rational numbers, and asserted that such a structure has a specified property over $\Q$ if and only if it has that property over $\Q_p$ for all $p$ and also for $\R$.  Later such results were generalized to arbitrary global fields.  Typically a local-global principle can be rephrased as asserting that a variety $Z$ has a 
rational point over a global field $F$ if it has a rational point over each completion $F_v$ of $F$. In many cases, $Z$ is a homogeneous space under some algebraic group. When $Z$ is a principal homogeneous space (i.e., a torsor) under an algebraic group $G$, the local-global principle can also be rephrased as asserting the vanishing of the Tate-Shafarevich set $\Sha(F,G):=\operatorname{ker}(H^1(F,G)\rightarrow \prod_v H^1(F_v,G))$.  More generally, the set $\Sha(F,G)$ is the obstruction to a local-global principle holding.

In the classical case of global fields, local-global principles have been proven in a number of contexts; e.g., for quadratic forms (Hasse-Minkowski), Brauer groups (Albert-Brauer-Hasse-Noether), and torsors under connected rational linear algebraic groups (\!\cite[Corollary 9.7]{sansuc}, \cite{Cher}).
Motivated by that classical case, analogs of the above results have in recent years been proven in the case of {\em semi-global fields}; i.e., function fields of transcendence degree one over a complete discretely valued field $K$ such as $\Q_p$.  (For example, see \cite{CTOP, hhk1, hhk2, hhktorsors, CTPS1, CTPS2, PPS}.)  In this situation, the geometry is richer, and there are various possible sets of overfields of $F$ with respect to which local-global principles and the corresponding Tate-Shafarevich sets can be studied. The current manuscript aims to clarify the relationships among those obstruction sets, and among the corresponding local-global principles for rational points on smooth varieties. 

In the global field case, the completions of $F$ are taken with respect to isomorphism classes of non-trivial absolute values on $F$.  For $F$ of equal characteristic (the global function field case), this is equivalent to taking the completions of $F$ with respect to the non-trivial valuations on $F$.  All of these valuations have rank one, and they are all discrete.  Moreover, they are in bijection with the closed points on the unique regular projective model $C$ of $F$, which is a curve over the finite base field of $F$.  Here the completion $F_P$ corresponding to a point $P \in C$ is the fraction field of the complete local ring $\wh{\mc O}_{C,P}$.

By contrast, the valuations on a semi-global field $F$
are not all of rank one, and those of rank one are not all discrete.
This leads to several choices for the set of completions, and therefore for the local-global principle being considered.  More precisely, write $T$ for the valuation ring of the complete discretely valued field $K$ over which $F$ is a function field.  Let 
$\Omega_F$ be the set of all non-trivial $T$-valuations on $F$ (i.e., those whose valuation ring contains $T$); let $\Omega_F^1$ be the set of all non-trivial rank one $T$-valuations on $F$; and let $\Omega^{\mathrm{dvr}}_F$ be the set of all discrete valuations on $F$ (whose valuation ring automatically contains $T$, by \cite[Corollary~7.2]{hhktorsors}).  Thus $\Omega^{\mathrm{dvr}}_F \subseteq \Omega^1_F \subseteq \Omega_F$.  Each of these sets then gives rise to a set of completions, and therefore to a local-global principle. If one focuses on torsors over $F$ under an algebraic group $G$, then one correspondingly obtains several versions of $\Sha$, which will be denoted by 
$\Sha_{\Omega_F}(F,G)$, $\Sha_{\Omega^1_F}(F,G)$, and $\Sha_{\mathrm{dvr}}(F,G)$.

In addition, there are many choices for a regular projective model $\mc X$ of $F$ over the valuation ring $T={\mc O}_K$ of the ground field $K$, each with a closed fiber $X \subset \mc X$; for each such model and point $P \in X$, the fraction field $F_P$ of $\wh{\mc O}_{\X,P}$ will not in general be the completion of $F$ with respect to a valuation.  (In fact, it will be such if and only if $P$ is the generic point of an irreducible component of $X$.)  Taking the set of overfields $\{F_P\}_{P \in X}$ for a given model $\mc X$ leads to yet another choice of local-global principle. For a torsor over $F$ under an algebraic group $G$, the corresponding obstruction will be denoted by $\Sha_X(F,G)$.  

In this paper, we prove the following (see Theorem~\ref{A}):

\begin{introtheorem*}\label{introA}
Let $G$ be a linear algebraic group over a semi-global field $F$.  Then 
for any regular projective model $\mc X$ of $F$ with closed fiber~$X$, 
\[
\Sha_X(F,G) \subseteq \Sha_{\Omega_F}(F,G) \subseteq \Sha_{\Omega^1_F}(F,G) \subseteq \Sha_{\mathrm{dvr}}(F,G).  
\]
Moreover, taking the direct limit over all such models $\mc X$, we have
\[
\varinjlim_{\X} \Sha_X(F,G) = \Sha_{\Omega_F}(F,G) = \Sha_{\Omega_F^1}(F,G).
\]
\end{introtheorem*}

More generally, we prove the corresponding result concerning local-global principles for the existence of rational points on smooth varieties over a semi-global field; see Theorem~\ref{B}.

The significance of these results is that one is usually interested in local-global principles with respect to valuations (of a suitable type), but these can be hard to study over semi-global fields.  On the other hand, by using patching methods (as in \cite{hhk1} and \cite{hhktorsors}), one can understand the obstruction to local-global principles with respect to the points on the closed fiber of a regular projective model quite explicitly, prove that local-global principles hold under certain hypotheses, and thereby obtain results about field invariants.  The results in this manuscript relate local-global principles with respect to points (or patches) to those with respect to rank one valuations (the natural notion in the context of Berkovich spaces; in particular see \cite{meh} for a Berkovich analog of the results in \cite{hhk1}).  We should note, however, that it remains open whether local-global principles with respect to the set of {\em discrete} valuations are equivalent to the other types of local-global principles.  \medskip

\section{Completions and henselizations of semi-global fields}
Let $K$ be a complete discretely valued field with ring of integers $T={\mc O}_K$ and residue field $k$. Let~$F$ be a {\em semi-global field} over $K$, i.e., a finitely generated field extension of transcendence degree one over $K$ in which $K$ is algebraically closed. (We do not assume that $F/K$ is separable.)  The aim of this section is to introduce certain overfields of $F$, defined via completions or henselizations, and to study their relationship.  This discussion will be important when we study rational points and local-global principles in the later sections.

We will consider two types of completions in this paper, associated to prime ideals (or points) and to valuations that need not be discrete.  
In general, for any commutative ring $A$, let $\cI=\{\eua_i\}_{i\in I}$ be a set of ideals of~$A$ such that for all $i',i''\in I$ there exists $i\in I$ with $\eua_i\subset\eua_{i'}\eua_{i''}$. The {\it $\cI$-completion} of $A$ is the projective limit $\,\wh A\,$ of the system~of~projections $A/\eua_{i'}\to A/\eua_i$,  $\eua_{i'}\subset\eua_i$.  Here the completion map $\hat\imath:A\to\wh A\,$ has $\ker(\hat\imath)=\cap_i\eua_i$.  

In particular, if $\frak p$ is a prime ideal of $A$ and we take $\eua_i = \frak p^i$ for $i \ge 1$, then we get the usual $\frak p$-{\em adic completion} of $A$.  
On the other hand, suppose we are given a field $L$ and a 
valuation $w:L \to \Gamma \cup \{\infty\}$, with $\Gamma$ an ordered abelian group (the {\em value group} of $w$), having valuation ring $\cO_w$ and maximal ideal $\frak m_w$.  In this situation, consider a set $\cI\!=\{\eua_i\}_{i\in I}$ of non-zero ideals of $\cO_w$ as above with $\cap_i\eua_i=(0)$.  (For example, we may take $\eua_i = \{a \in A\,|\, w(a)\geq i\}$ for all $i>0$ in $\Gamma$.)  The $\cI$-completion $\wh\cO_w$ of $\cO_w$ is independent of the choice of such a set $\cI\!$; and $\wh\cO_w$ is a valuation ring of its fraction field $L_w$, with valuation ideal $\wh\eum_w=\eum_w\wh\cO_w$. One calls 
$\wh\cO_w$ and $L_w$ the {\em $w$-adic completions} of $\cO_w$ and $L$, respectively.  Note that $\cO_w$ is a local ring with maximal ideal $\frak m_w$; however, the $w$-adic completion of $\cO_w$ is not in general the same as the $\frak m_w$-adic completion of $\cO_w$ (in fact, they are the same if and only if $w$ is a discrete valuation; otherwise $\cap_i \frak m_w^i$ is a non-zero prime ideal).  See \cite[Chapter VI]{Bou} and \cite{englerprestel} for more about valuations and their completions.

Throughout this paper, for a semi-global field $F$ we will restrict attention to valuations that lie in $\Omega_F$, the set of non-trivial $T$-valuations on $F$ (i.e., those whose valuation ring contains $T$).  By \cite[Corollary~7.2]{hhktorsors}, all discrete valuations lie in that set.  For each valuation $v \in \Omega_F$, we let $F_v$ denote the $v$-adic completion of $F$, as in the general discussion above. This gives a first collection of overfields of~$F$.

We next define overfields that are obtained from the geometry of curves with function field~$F$.
A {\em normal model} of $F$ is a normal integral $T$-scheme $\X$ with function field $F$ that is flat and projective over $T$ of relative dimension one. If in addition the scheme $\X$ is regular we say that it is a {\em regular model}. The {\em closed fiber} of $\X$ is $X:=\left(\X\times_Tk\right)^{\mathrm{red}}$. 

Let $\X$ be a regular model of $F$ with closed fiber $X$.
The points of $X$ consist of the closed points of $\X$ and the generic points of the irreducible components of $X$.  For each point $P$ in~$X$, let $\mathcal{O}_{\X,P}$ be the local ring of $\X$ at $P$ with maximal ideal $\mathfrak{m}_P$, let $\wh{\mc O}_{\X,P}$ be its $\mathfrak{m}_P$-adic completion, and let $F_P$ be the fraction field of $\wh{\mc O}_{\X,P}$.  These give the second set of overfields of~$F$.

An essential first step in understanding the relationship between the various local-global principles 
(or their obstructions) is to study the relationship between the overfields $F_v$ and $F_P$ of $F$ introduced above. In order to do so, we need to recall a few basic facts about henselian local rings and henselization. 

Recall that a local ring $\mathcal{O}$ with maximal ideal $\mathfrak{m}$ and residue field $\kappa$
is {\em henselian} if Hensel's lemma holds in ${\mathcal O}$. That is, given a monic polynomial $f \in \mathcal{O}[x]$ with image $\bar{f} \in \kappa[x]$, any simple root $\bar a\in\kappa$ of $\bar{f}$ lifts to a simple root $a\in\cO$ of $f$.
A {\em henselization} of a local ring ${\mathcal O}$ with maximal ideal $\mathfrak{m}$ is a henselian local ring  $\mathcal{O}^h$ with the following properties: $\mathcal{O}^h$ is a direct limit of {\'e}tale $\mathcal{O}$-algebras,  $\mathfrak{m}\mathcal{O}^h$ is the maximal ideal of~$\mathcal{O}^h$, and $\mathcal{O}/\mathfrak{m} = \mathcal{O}^h/\mathfrak{m}\mathcal{O}^h$. A henselization exists for any local ring and it is unique up to isomorphism (see, for example, ~\cite[Theorems 32.28 and 32.29]{warnerfields}). 
If ${\mathcal O}$ is an integrally closed local domain with fraction field $L$, then the fraction field of $\mathcal{O}^h$ is correspondingly a direct limit of finite separable extensions of $L$; see also ~\cite[p.~48]{blr}.  In particular, 
for $\X$ and $P$ as above, we may consider the henselization $\mathcal{O}_{\X,P}^h$ of the local ring $\mathcal{O}_{\X,P}$ of $\X$ at $P$.
We may also consider the henselization $\mathcal{O}_w^h$ of a valuation ring $\mathcal{O}_w$ associated to a valuation $w$ on some field $L$. 

If $w$ is a valuation on a field $L$, then $w$ and $L$ are said to be {\em henselian} if the valuation ring~$\mathcal{O}_w$ is henselian. 
Equivalently, $w$ is henselian if it has a unique extension to each algebraic field extension of~$L$ \cite[Definition~A.13]{tigwad}.
Other equivalent formulations of the henselian property for valuation rings can be found in, e.g., \cite[Theorem A.14]{tigwad}.

In the above situation, the fraction field $L_w^h$ of $\mathcal{O}_w^h$ is the {\em henselization} of $L$ with 
respect to $w$.  That is, it is the unique (up to unique isomorphism) 
valued field extension of $L$ whose valuation is henselian and 
extends $w$, and which is universal with these properties.

The {\em rank} of a valuation is the Krull dimension of its valuation ring.  
A valuation is of rank one if and only if there is an order-preserving embedding of its value group into $\R$.
In the situation of semi-global fields $F$, it follows from Abhyankar's inequality (see~\cite{abh} or \cite[Theorem 3.4.3]{englerprestel})  
that every valuation $v \in \Omega_F$ has either rank one or rank two. 
Moreover, the rank two valuations are of the form $v = v_2 \circ v_1$, where $v_1 \in \
\Omega^1_F$, where $v_2$ is a rank one valuation of the residue field of $v_1$, and where composition is in the sense of valuations; cf.~\cite[Chapter I, \S 4]{schilling}.  (See \cite[Appendix A]{popstix}
for a classification of $T$-valuations on semi-global fields; although smoothness was assumed there, it was not essential.) 

In studying henselizations and completions of semi-global fields, we distinguish between two cases, based on the rank of the valuation.  For a rank one valuation $w$ on an arbitrary field~$L$, the completion $L_w$ of $L$ with respect to $w$ is henselian (e.g., see \cite[Proposition 1.2.2]{englerprestel}), and the henselization of $L$ with respect to $w$ is the separable closure of $L$ in its completion (this follows via \cite[Theorem 32.19]{warnerfields}).  
For valuations of rank two on a semi-global field, there is the following result.

\begin{lemma}\label{refine} 
Let $F$ be a semi-global field and let $v \in \Omega_F$ have rank two. As above, write $v = v_2 \circ v_1$ as a composition of a valuation $v_1$ of $F$ of rank one, and a valuation $v_2$ of the residue field of $v_1$. Then
\begin{enumerate}
\item \label{lem:rank2 compl}
the completions $F_v = F_{v_1}$ are equal; and
\item \label{lem:rank2 hens}
the henselizations satisfy $F^h_{v_1} \subseteq F^h_{v}$.
\end{enumerate}
\end{lemma}

\begin{proof}
Part~(\ref{lem:rank2 compl}) is proven in \cite[Lemma 2.2.2 and Corollary 2.2.27]{aschDH}. By \cite[Proposition A.31]{tigwad}, there is a henselian valuation $v_1'$ on 
$F_v^h$ that extends $v_1$.  Thus $F_{v_1}^h$ may be identified with the 
relative henselization of $F$ in $F_v^h$, with respect to $v_1$ and $v_1'$, proving~(\ref{lem:rank2 hens}).
\end{proof}

In the situation of the introduction, Proposition~\ref{fpfv} below will provide an essential link between the fields $F_v$ for $v \in \Omega_F$ and the fields $F_P$ for $P \in X$.  First we state two lemmas.

\begin{lemma} \label{lem:injective}
Let $\X$ be a regular model of a semi-global field $F$, and let $Q$ be a point of $\X$.  If $I$ is an ideal of $\wh{\mc O}_{\X,Q}$ such that $I \cap {\mc O}_{\X,Q} = (0)$, then $I=(0)$.
\end{lemma}

\begin{proof}
The assertion is trivial if $Q$ is the generic point.  It is also clear if $Q$ is a point of codimension one, since then ${\mc O}_{\X,Q}$ is a discrete valuation ring, and contraction induces a bijection between the ideals of the completion $\wh{\mc O}_{\X,Q}$ and the ideals of ${\mc O}_{\X,Q}$.  So we now assume that $Q$ has codimension two; i.e.,\ $Q$ is a closed point of $\X$, lying on the closed fiber $X$ of $\X$.  

First consider the case that $\X = \PP^1_T$ and that $Q$ is the origin on $X = \PP^1_k$.  Let $g$ be an element of the ideal $I \subset \wh{\mc O}_{\X,Q}$.  By the Weierstrass Preparation Theorem (e.g., see~\cite[Proposition~VII.3.9.6]{Bou}), there exist an element $f \in F$ and a unit $u \in \wh{\mc O}_{\X,Q}^\times$ such that $g=fu$.  Thus $f = gu^{-1}$ lies in 
$\wh{\mc O}_{\X,Q} \cap F = {\mc O}_{\X,Q}$ and in $I$.  So $f \in I \cap {\mc O}_{\X,Q}$ is equal to $0$ by hypothesis.  Hence $g=fu=0$, concluding the proof of this case.  

For a general $\X$, there exists a finite morphism $\X \to \X' := \PP^1_T$ that takes $Q$ to the origin $P \in \PP^1_k \subset \X'$ (see~\cite[Proposition~6.6]{HH}).    
The inclusion $\mc O_{\X',P} \subseteq {\mc O}_{\X,Q}$ induces a morphism $\wh{\mc O}_{\X',P} \to \wh{\mc O}_{\X,Q}$, which is finite since $\X \to \X'$ is finite, and is injective because the $\frak m_Q$-adic topology on ${\mc O}_{\X,Q}$ restricts to the $\frak m_P$-adic topology on ${\mc O}_{\X',P}$.
The contraction of~$I$  to $\mc O_{\X',P} \subset {\mc O}_{\X,Q}$ is trivial, hence so is the contraction to $\wh{\mc O}_{\X',P}$ by the case of $\X = \PP^1_T$ above.  Since $\wh{\mc O}_{\X,Q}$ is finite over $\wh{\mc O}_{\X',P}$, the ideal~$I$ itself is trivial.
\end{proof}

We note that as an alternative argument for a general $\X$, we could use \cite[Proposition~3.4]{hhkWeier} to reduce to the case where the Weierstrass Preparation Theorem \cite[Theorem~3.1(c)]{hhkWeier} applies, and then use that result to prove Lemma~\ref{lem:injective}.

\smallskip

Recall that the \emph{center} of a valuation $w$ on the function field of a separated integral scheme $\cZ$ is a point $Q$ on $\cZ$ such that $\mathcal{O}_w$ contains the local ring $\mathcal{O}_{\cZ,Q}$, and such that $\mathfrak{m}_{\cZ,Q} = \mathfrak{m}_v \cap \mathcal{O}_{\cZ,Q}$.
If a center exists on $\cZ$ then it is unique, and if $w$ is non-trivial then the center is not the generic point of $\cZ$.

The next lemma is well known to experts, but there does not seem to be a good reference for it in the literature.  For the sake of completeness, we include a short proof.  (This lemma remains true more generally for proper integral schemes by the valuative criterion for properness, but we do not need that stronger form here.)  

\begin{lemma}\label{lem:center}
Let $\cZ$ be a projective integral scheme over a ring $A$, and let $w$ be an $A$-valuation on the function field of $\cZ$ (i.e., the valuation ring of $w$ contains $A$).  Then $w$ has a center on $\cZ$.
\end{lemma}

\begin{proof}
Write $\cZ=\Proj S$, where $S$ is a graded $A$-algebra generated by finitely many elements 
$x_0,\dots,x_n$ of degree one.  Choose $i$ such that $w(x_i/x_0)=\min_{j>0} w(x_j/x_0)$.  Thus $w(x_j/x_i)\ge 0$ for all $j$. Hence $R:=A[x_j/x_i]_{0\le j\le n}\subset \cO_w$, 
with $\Spec(R)=D_{x_i}^+$, the open subset of~$\cZ$ where $x_i \ne 0$.  The prime ideal $\eum_w \cap R$ of $R$ defines a point $Q \in D_{x_i}^+ \subset \cZ$ such that $\cO_{\cZ,Q}\subset\cO_w$ and $\eum_{\cZ,Q}=\eum_w\cap\cO_Q$.  That is, $Q$ is the center of $w$ on $\cZ$.
\end{proof}

\begin{proposition}\label{fpfv}
Let $F$ be a semi-global field over a complete discretely valued field $K$, and let $\X$ be a regular model of $F$  with closed fiber~$X$. 
For every valuation $v \in \Omega_F$ there exists a point $P \in X$ (not necessarily closed) 
such that $F_P \subseteq F_v$.
\end{proposition}

\begin{proof}
First suppose that $v$ is of rank one.  
By Lemma~\ref{lem:center}, $v$ has a center $Q$ on $\X$.  
Then $c := \min \{v(f) \mid f\in \mathfrak{m}_{\X,Q} \} >0$ since $\mathfrak{m}_{\X,Q}$ is finitely generated; and 
\[\mathfrak{m}_{\X,Q}^i \subseteq \eua_i := \{a \in \mathcal{O}_v \,|\, v(a)\geq ic\}.\] 
Since $v(F) \subseteq \mathbb{R}$, we have $\cap_i \eua_i~=~(0)$.
Hence $\mathcal{O}_{\X,Q} \hookrightarrow \mathcal{O}_v$ is a continuous map of topological rings, and there is an induced map $\widehat{\mathcal{O}}_{\X,Q} \to \widehat{\mathcal{O}}_v$ between the respective completions.
This map is injective, as can be seen by applying Lemma~\ref{lem:injective} to the kernel.  Therefore it induces an inclusion of fraction fields $F_Q~\hookrightarrow~F_v$.  

If $Q$ lies on $X$, then $P=Q$ satisfies the required condition.  On the other hand, if $Q$ does not lie on $X$, then $Q$ is a codimension one point of the generic fiber $\X_K$ of $\X$, and 
there is a discrete valuation $v$ on $F$ corresponding to $Q$.  The closure of $Q$ in $\X$ meets $X$ at a closed point $P$, and the containment ${\mc O}_{\X,P} \subset {\mc O}_{\X,Q}$ induces an inclusion $\wh {\mc O}_{\X,P} \subset \wh{\mc O}_{\X,Q}$.  Namely, this is clear if $\X$ is the projective line over the valuation ring of $K$ and $P$ is the origin on the closed fiber, and one can reduce to that case as in the proof of Lemma~\ref{lem:injective} above.  Thus $F_P \subset F_Q = F_v$.

Next, suppose that $v$ is of rank two, so that we may write $v = v_2 \circ v_1$ as in the discussion before Lemma~\ref{refine}, for some rank one valuation $v_1 \in \Omega^1_F$.  By the previous case, $F_{v_1}$ contains $F_P$ for some point $P \in X$.  But $F_v=F_{v_1}$
by Lemma~\ref{refine}(\ref{lem:rank2 compl}), and so $F_v$ contains $F_P$.
\end{proof}

\section{Local points on $F$-varieties}\label{sec:local}

In this section, we study varieties over a semi-global field $F$ that have rational points over field extensions $F_v$ (for $v \in \Omega_F$) or $F_P$ (for $P$ in the closed fiber $X$ of a regular model $\X$).  We begin by stating some consequences of the results in the previous section.

\begin{proposition}\label{rk1=all}
Let $Z$ be an $F$-variety.
Then $Z(F_v) \neq \emptyset$ for \emph{all} valuations $v \in \Omega_F$, if and only if $Z(F_v) \neq \emptyset$ for all valuations $v \in \Omega^1_F$. 
\end{proposition}

\begin{proof}
The forward implication is immediate, so it suffices to prove the reverse implication.
As above, a valuation in $\Omega_F$ that is not of rank one is a rank two valuation
$v = v_2 \circ v_1,$ with each $v_i$ of rank one.
By Lemma \ref{refine}(\ref{lem:rank2 compl}), the completion $F_v$ at $v$ is equal to the completion $F_{v_1}$ at the rank one valuation $v_1$. Hence $Z(F_v) = Z(F_{v_1})$ is nonempty.
\end{proof}

\begin{proposition}\label{inc}
Let $\X$ be a regular model of $F$ with closed fiber $X$. Let $Z$ be an $F$-variety. If $Z(F_P) \neq \emptyset$ for all $P \in X$, then $Z(F_v) \neq \emptyset$ for all valuations $v \in \Omega_F$.
\end{proposition}

\begin{proof}
If $v$ is a valuation in $\Omega_F$, then $F_v$ contains a field $F_P$ for some $P \in X$ by Proposition~\ref{fpfv}.
Hence $Z(F_P) \subseteq Z(F_v)$, and so $Z(F_v) \neq \emptyset$. 
\end{proof}

The next proposition allows us to deduce the existence of $F_v^h$-points on a variety from $F_v$-points. This will be important in proving the main assertion in this section, Theorem~\ref{rk1inpt}.  

\begin{proposition}\label{AA}
Let $Z$ be a smooth $F$-variety.
\begin{enumerate}
\item \label{AAa}
Let $v \in \Omega^1_F$. 
If $Z(F_v) \neq \emptyset$, then $Z(F^h_v) \neq \emptyset$.
\item \label{AAb}
Let $\X$ be a regular model of $F$ and let $P$ be a closed point on its closed fiber. If $Z(F_P)~\neq~\emptyset$, then $Z(\mathrm{Frac}(\mathcal{O}_{\X,P}^h)) \neq \emptyset$.
\end{enumerate}
\end{proposition}

\begin{proof}
Part~(\ref{AAa}) of the proposition follows from \cite[Proposition 3.5.2]{GGMB}, in the case of a smooth variety.  (Equivalently, it follows from the generalized Implicit Function Theorem (see \cite[Theorem 9.2]{greenpoproquette}), after choosing an \'etale morphism from an affine open subset to some $\mathbb{A}^d_F$, using smoothness.) 

For part~(\ref{AAb}),
recall that $F_P = \mathrm{Frac}(\wh{\mc O}_{\X,P})$, where $\wh{\mc O}_{\X,P}$ is the $\mathfrak{m}_P$-adic completion of $\mathcal{O}_{\X,P}$.
Let $U = \mathrm{Spec}(B)$ be an affine open neighborhood of $P$ in $\X$. 
Then $B$ is a $T$-algebra of finite type, with a maximal ideal $\mathfrak{m}'_P$ corresponding to $P$, and with fraction field $F$.
The henselization of the localization of $B$ at $\mathfrak{m}'_P$ is $\mathcal{O}_{\X,P}^h$.

After replacing $Z$ by an affine open subset that contains an $F_P$-point, we may assume that $Z$ is an affine $F$-variety, say $\mathrm{Spec}(F[t_1, t_2, \ldots, t_k]/(f_1, f_2, \ldots, f_n))$.  Since $F_P = \wh{\mc O}_{\X,P} \otimes_{{\mc O}_{\X,P}} F$ by Lemma~\ref{lem:injective}, after clearing denominators we may assume that the coefficients of the polynomials $f_i$ lie in $B$; that the $B$-variety $W := \mathrm{Spec}(B[t_1, t_2, \ldots, t_k]/(f_1, f_2, \ldots, f_n))$ has generic fiber $Z$; and that $W$ has an $\wh{\mc O}_{\X,P}$-point whose general fiber is the given $F_P$-point of $Z$.
By the existence of this $\wh{\mc O}_{\X,P}$-point, there is a solution in $\wh{\mc O}_{\X,P}$ to $f_1 = \ldots = f_n = 0$.
Since the complete discrete valuation ring $T$ is excellent, the Artin Approximation Theorem (see \cite[Theorem~1.10]{artin}) applies; and so there is also a solution in $\mathcal{O}_{\X,P}^h$ to $f_1 = \ldots = f_n = 0$. 
That is, $W$ has an $\mathcal{O}_{\X,P}^h$-point, so $Z$ has a $\mathrm{Frac}(\mathcal{O}_{\X,P}^h)$-point.
\end{proof} 

The next lemma is the key ingredient to the main theorem of this section.

\begin{lemma}\label{lem:chain}
Consider an infinite sequence of regular models of $F$  
\begin{equation*}
\X = \X_0 \leftarrow \X_1 \leftarrow \X_2 \leftarrow \dots
\end{equation*}
and non-empty finite sets $\mathcal{P}_i$ of closed points on the respective closed fibers $X_i$ of the models $\X_i$, where each $\X_i$ is obtained by blowing up $\X_{i-1}$ at the ideal defined by the set $\mathcal{P}_{i-1}$. 
Let $P_0, P_1, P_2, \ldots$ be an infinite sequence of points $P_i \in \mathcal{P}_i$ such that $P_{i+1}$ maps to $P_{i}$ for all $i \geq 0$, and consider the direct limit $\mathcal{O} := \varinjlim_i \mathcal{O}_{\X_i,P_i}$. 
Then $\mathcal{O}$ is a valuation ring on $F$, with valuation $v \in \Omega_F$ (i.e., $\mathcal{O} = \mathcal{O}_v$).
Moreover, 
 \begin{equation}\label{eq:limhens}
\mathcal{O}^h_v = \varinjlim_i \mathcal{O}^h_{\X_i, P_i}.
\end{equation}
In particular, $F^h_v = \varinjlim_{i} \mathrm{Frac}(\mathcal{O}^h_{\X_i, P_i}) \subseteq \varinjlim_{i} F_{P_i}$.
\end{lemma}

\begin{proof}
We first prove that $\mathcal{O}$ is a valuation ring of $F$; equivalently, for any $f \in F^{\times}$, either $f \in \mathcal{O}$ or $f^{-1} \in \mathcal{O}$.
For this, it suffices to show that for any fixed $f \in F^{\times}$, there exists an index~$i$ such that either $f \in \mathcal{O}_{\X_i,P_i}$ or $f^{-1} \in \mathcal{O}_{\X_i,P_i}$. 
This condition does in fact hold, by \cite[Theorem 26.2]{lipman}, which states that there does not exist an infinite sequence of points $P_i \in X_i \subseteq \X_i$ as in the hypothesis of this lemma, for which each $P_{i}$ is a point of indeterminacy for $f$.  So $\mathcal{O}$ is a valuation ring of $F$.  Its valuation $v$ lies in $\Omega_F$, because each local ring $\mathcal{O}_{\X,P_i}$ contains $T$ by definition, and hence so does $\mathcal{O}$.

To prove~\eqref{eq:limhens}, it suffices to show that
$R := \varinjlim_{i} \mathcal{O}^h_{\X_i, P_i}$
satisfies the properties that characterize the henselization of $\mathcal{O}$ with 
respect to $v$.  This ring is henselian because it is a direct limit of 
henselian rings.  For each $(\X_i,P_i)$, the ring $\mathcal{O}^h_{\X_i, P_i}$ is 
a direct limit of \'etale $\mathcal{O}_{\X_i, P_i}$-algebras $A_{\X_i,P_i,j}$, 
and so $R$ is the direct limit of the \'etale $\mathcal{O}_v$-algebras 
$\mathcal{O}_v A_{\X_i,P_i,j}$, as $i,j$ vary.  Since 
$\frak{m}\mathcal{O}^h_{\X_i, P_i}$ is the unique maximal ideal of the 
henselization $\mathcal{O}^h_{\X_i, P_i}$ of $\mathcal{O}_{\X_i, P_i}$ for all $i$, it 
follows that
$\frak{m}R$ is the unique maximal ideal of $R$.  Similarly, since
$\mathcal{O}^h_{\X_i, P_i}/\frak{m}\mathcal{O}^h_{\X_i, P_i} = 
\mathcal{O}/\frak{m}$ for all~$i$, it follows that $R/\frak{m}R  = 
\mathcal{O}/\frak{m}$.  So \eqref{eq:limhens} holds.
The last statement follows since construction of fraction fields commutes with direct limits and respects inclusions, and since $\mathcal{O}_{\X_i,P_i} \subseteq \mathcal{O}^h_{\X_i,P_i} \subseteq \wh{\mc O}_{\X_i,P_i}$ for all $i$. 
\end{proof}

As above, $F$ is a semi-global field over a complete discretely valued field $K$.

\begin{theorem}\label{rk1inpt}
Let $Z$ be a smooth $F$-variety that has an $F_v$-point for all $v \in \Omega^1_F$. Then there exists a regular model $\X$ of $F$ such that 
$Z(F_P)\neq \emptyset$ for all $P$ in the closed fiber $X$ of $\X$.
\end{theorem}

\begin{proof}
Fix a regular model $\X_0$ of $F$, and let $X_0$ denote its closed fiber. 
Let $Y_0$ be an irreducible component of $X_0$. Then the generic point $\eta$ of $Y_0$ corresponds to a valuation $v$ of rank one (centered on $\eta$). So by assumption, $Z(F_v) = Z(F_{\eta})$ is nonempty. Let $t$ be a uniformizer of $K$. For a nonempty affine open subset $U \subseteq Y_0$ that does not meet any other irreducible component of $X_0$, let $F_U$ denote the fraction field of the $t$-adic completion of the ring of rational functions on $\X_0$ that are regular along $U$ (see \cite[Notation~3.3]{hhk1}). By \cite[Proposition 5.8]{hhktorsors}, there exists such an open subset $U$ for which $Z(F_U) \neq \emptyset$.  (This could also be deduced from Proposition~\ref{AA}(\ref{AAa}) above.)
If $P\in U$, then $F_U\subseteq F_P$ (loc.\ cit.), 
and thus $Z(F_P)\neq \emptyset$ for such $P$.  Since $Y_0\setminus U$ is finite, there are at most finitely many points $P\in Y_0$ for which the set $Z(F_P)$ is empty. 
Ranging over the finitely many components of $X_0$, we obtain a finite (possibly empty) set $\mathcal{P}_0 \subseteq X_0$ consisting of exactly those points $P\in X_0$ for which $Z(F_{P}) = \emptyset$.

Let $\X_1$ be the blowup of $\X_0$ at all points of $\mathcal{P}_0$ and let $X_1$ be its closed fiber.
By the same argument as above, there exists a finite (possibly empty) set $\mathcal{P}_1 \subseteq X_1$ such that $Z(F_P) = \emptyset$ exactly for $P \in \mathcal{P}_1$. Let $\X_2$ be the blowup of $\X_1$ at all points of $\mathcal{P}_1$, etc. 
This process yields a chain of models $\X_0 \leftarrow \X_1 \leftarrow \X_2 \leftarrow \ldots$ and corresponding sets $\mathcal{P}_0, \mathcal{P}_1, \mathcal{P}_2, \ldots$ of closed points on their respective closed fibers. 

\smallskip

We claim that the disjoint union $\mathcal{P}:= \bigcup_{i \geq 0} \mathcal{P}_i$ is finite. This claim immediately implies that the chain of models $\X_0 \leftarrow \X_1 \leftarrow \X_2 \leftarrow \ldots$ terminates in some model $\X_M=:\X$ for which $\mathcal{P}_M = \emptyset$. That is, the model $\X$ satisfies $Z(F_P) \neq \emptyset$ for all points $P$ in its closed fiber $X$, as required.

\smallskip

It remains to prove the claim. Suppose to the contrary that $\mathcal{P}$ is infinite. For $j\geq i$, we say that a point $Q$ on $X_j$ {\em lies over} a point $P$ on $X_i$ if $Q$ maps to $P$ under the sequence of blowups $\X_i\leftarrow \cdots \leftarrow \X_j$. Note that each point of $\mathcal{P}_1$ lies over some point of $\mathcal{P}_0$, and similarly each point of $\mathcal{P}_{i+1}$ lies over some point of $\mathcal{P}_i$, for all $i\geq 1$.
Since $\mathcal{P}$ is infinite and ${\mathcal P}_0$ is finite, there exists a point $P_0\in \mathcal{P}_0$ such that infinitely many points in $\mathcal{P}$ lie over $P_0$. Inductively, if $P_i \in {\mathcal P}_i$ is a point such that infinitely many points in $\mathcal{P}$ lie over $P_i$, there is a point $P_{i+1}\in {\mathcal P}_{i+1}$ lying over $P_i$, such that infinitely many points in $\mathcal{P}$ lie over $P_{i+1}$. This defines an infinite sequence of points $P_1, P_2, \ldots$, since $\mathcal{P}$ is infinite.  Here $\mathcal{O}_{\X_i, P_i} \subseteq \mathcal{O}_{\X_{i+1}, P_{i+1}}$, 
$\wh{\mathcal{O}}_{\X_i, P_i} \subseteq \wh{\mathcal{O}}_{\X_{i+1}, P_{i+1}}$, and
$F_{P_i} \subseteq F_{P_{i+1}}$. 

Consider the direct limit $\mathcal{O} := \varinjlim_{i} \mathcal{O}_{\X_i, P_i}$. By Lemma \ref{lem:chain}, the ring $\mathcal{O}$ is a valuation ring with respect to some valuation $v$ on $F$. If $v$ has rank one, then $Z(F_v) \neq \emptyset$ by assumption, and consequently $Z(F_v^h)\neq \emptyset$ by Proposition~\ref{AA}(\ref{AAa}). If $v$ has rank two, then Lemma~\ref{refine} implies that there exists a rank one valuation $v_1$ of $v$ for which $F_{v} = F_{v_1}$  and $F_{v_1}^h \subseteq F_v^h$, so again $Z(F_v)=Z(F_{v_1})\neq \emptyset$, and hence $\emptyset\neq Z(F_{v_1}^h)\subseteq Z(F_v^h)$ by Proposition~\ref{AA}(\ref{AAa}).  
By Lemma~\ref{lem:chain}, $F_v^h\subseteq \varinjlim_i F_{P_i}$, and thus 
\[
\emptyset\neq Z(F_v^h) \subseteq Z (\varinjlim_i F_{P_i} ) = Z(\cup_i F_{P_i})= \cup_i Z(F_{P_i}) = \emptyset,
\]
where the equalities hold because $F_{P_i}\subseteq F_{P_{i+1}}$. This contradiction proves the claim.
\end{proof}

\begin{remark}
\begin{enumerate}
\item The model given by Theorem~\ref{rk1inpt} also satisfies the a priori stronger assertion that $Z(\mathrm{Frac}(\mathcal{O}_{\X,P}^h)) \neq \emptyset$ for all points $P$ in the closed fiber $X$ of $\X$, by Proposition~\ref{AA}(\ref{AAb}).
\item In Theorem~\ref{rk1inpt}, instead of assuming that $Z$ is smooth, we could assume the existence of a smooth $F_v$-point for every~$v$, since we could apply the theorem to the smooth locus of~$Z$.
\end{enumerate}
\end{remark}

\section{Local-global principles for rational points}

In this section, we use the results of Section~\ref{sec:local} to deduce results on local-global principles. Subsection~\ref{ssec:LGPvar} treats local-global principles for varieties, while Subsection~\ref{ssec:lgptors} concerns local-global principles for torsors.

As before, $F$ is a semi-global field over a complete discretely valued field $K$ with ring of integers $T$. 
Let $\Omega_F$ denote the set of valuations on $F$ whose valuation ring contains $T$, write $\Omega^1_F$ for the subset of valuations of rank one whose valuation ring contains $T$, and write $\Omega_F^{\mathrm{dvr}}$ for the subset of discrete valuations. For a regular model $\X$ of $F$ with closed fiber $X$, let $\Omega_{\X}$ denote the set of points $P$ of $X$.

\subsection{Local-global principles for $F$-varieties}\label{ssec:LGPvar}

The results in Section~\ref{sec:local} immediately give the following theorem:

\begin{theorem}\label{B}
Let $Z$ be a smooth variety over a semi-global field $F$.  
Then the following are equivalent:
\begin{enumerate}
\item\label{Ba} There is a regular model $\X$ of $F$ such that $Z(F_P) \neq \emptyset$ for all points $P$ in its closed fiber~$X$.
\item\label{Bb} $Z(F_v) \neq \emptyset$ for all valuations $v \in \Omega^1_F$.
\item\label{Bc} $Z(F_v) \neq \emptyset$ for all valuations $v \in \Omega_F$.
\end{enumerate}
\end{theorem}

\begin{proof}
The assertion~(\ref{Ba}) implies~(\ref{Bc}) by Proposition~\ref{inc}. Trivially~(\ref{Bc}) implies~(\ref{Bb}). The fact that~(\ref{Bb}) implies~(\ref{Ba}) holds by Theorem~\ref{rk1inpt}. 
\end{proof}

We say that a class $\cC$ of $F$-varieties \emph{satisfies a local-global principle} with respect to a set $\{ F_i \}_i$ of overfields of $F$ if every $Z$ in $\cC$ has the property that if
$Z(F_i) \neq \emptyset$ for all $F_i$  then $Z(F) \neq \emptyset$.  In this language, Theorem~\ref{B} immediately yields the following corollary.

\begin{corollary} \label{cor:lgp equiv}
Let $\cC$ be a class of smooth varieties over a semi-global field $F$.  Then 
$\cC$ satisfies a local-global principle with respect to $\{F_v \mid v \in \Omega_F\}$ (or equivalently, with respect to $\{F_v \mid v \in \Omega^1_F\}$)
if and only if 
$\cC$ satisfies a local-global principle with respect to $\{F_P \mid P \in X\}$ for every regular model $\X$ of $F$ with closed fiber $X$.
\end{corollary}

\begin{example} Let $G$ be a rational connected linear algebraic group over the semi-global field $F$, and let $\cC$ be the class of transitive homogeneous $G$-spaces $Z$ over $F$ (i.e., $G(E)$ acts transitively on $Z(E)$ for every overfield $E$ of $F$).  Then for every regular model $\X$ of $F$, say with closed fiber $X$, the class $\cC$ satisfies a local-global principle with respect to $\{F_P \mid P \in X\}$ (see \cite[Theorem~9.1]{hhktorsors}).  
As a consequence of Corollary~\ref{cor:lgp equiv}, it then follows that the class $\cC$ also satisfies a local-global principle with respect to the set $\{F_v \mid v \in \Omega^1_F\}$.
\end{example}

\subsection{Local-global principles for torsors}\label{ssec:lgptors}

Let $G$ be a linear algebraic group (i.e., a smooth affine group scheme of finite type) over~$F$. 
We now define and compare several obstruction sets to the existence of global $F$-points on torsors under $G$. 

Recall that each $G$-torsor over $F$ is represented by a class in the pointed set $H^1(F,G)$, and that this class is trivial if and only if the torsor has an $F$-point.  As discussed in the introduction, we have 
the following obstruction sets to local-global principles:
\[
\Sha_{\Omega_F}(F,G) = \ker \bigl( H^1(F,G) \to \prod_{v \in \Omega_F} H^1(F_v,G) \bigl),
\]
\[
\Sha_{\Omega^1_F}(F,G) = \ker \bigl( H^1(F,G) \to \prod_{v \in \Omega^1_F} H^1(F_v,G) \bigl),
\] 
\[
\Sha_{\mathrm{dvr}}(F,G) = \ker \bigl( H^1(F,G) \to \prod_{v \in \Omega^{\mathrm{dvr}}_F } H^1(F_v,G) \bigl).
\]
Here the {\em kernel} of a map of pointed sets is by definition the preimage of the trivial element.

Since $\Omega^{\mathrm{dvr}}_F \subseteq \Omega^1_F \subseteq \Omega_F$, the above obstruction sets are related by containments as well:
\[\Sha_{\Omega_F}(F,G) \subseteq \Sha_{\Omega^1_F}(F,G) \subseteq \Sha_{\mathrm{dvr}}(F,G).\] 

Finally, for any regular model ${\mathcal X}$ of $F$ with closed fiber~$X$, we let
\[
\Sha_X(F,G) = \ker \bigl( H^1(F,G) \to \prod_{P \in X} H^1(F_P,G) \bigl).
\]
By the previous subsection, we obtain:
\begin{theorem}\label{A}
Let $G$ be a linear algebraic group over a semi-global field $F$.  Then 
\[
\Sha_X(F,G) \subseteq \Sha_{\Omega_F}(F,G)\]
for any regular model ${\mc X}$ of $F$.  
Moreover, taking the direct limit over all such models $\mc X$, we have
\[
\varinjlim_{\X} \Sha_X(F,G) = \Sha_{\Omega_F}(F,G) = \Sha_{\Omega^1_F}(F,G).
\]
\end{theorem}

\begin{proof}
If $Z$ is a $G$-torsor over $F$, then $Z$ is smooth (because $G$ is). Proposition~\ref{inc} shows that $\Sha_X(F,G) \subseteq \Sha_{\Omega_F}(F,G)$ for any regular model $\X$ of $F$, 
and so $\varinjlim_{\X} \Sha_X(F,G) \subseteq \Sha_{\Omega^1_F}(F,G)$.  The asserted equalities follow by applying Theorem~\ref{B} to $Z$.  
\end{proof}

It is an interesting open problem to understand the relationship between $\Sha_{\Omega^1_F}(F,G)$ and $\Sha_{\mathrm{dvr}}(F,G)$; in particular, whether the inclusion is an equality.

\smallskip

{\small \noindent Author information:

\smallskip

\noindent David Harbater: Department of Mathematics, University of Pennsylvania, Philadelphia, PA 19104-6395, USA. E-mail: {\tt harbater@math.upenn.edu}

\smallskip

\noindent Julia Hartmann:  Department of Mathematics, University of Pennsylvania, Philadelphia, PA 19104-6395, USA. E-mail: {\tt hartmann@math.upenn.edu}

\smallskip

\noindent Valentijn Karemaker:  Mathematical Institute, Utrecht University, P.O. Box
80010, 3508 TA Utrecht, the Netherlands and Department of Mathematics,
Stockholm University, SE 10691 Stockholm, Sweden.  E-mail: {\tt V.Z.Karemaker@uu.nl}

\smallskip

\noindent Florian Pop: Department of Mathematics, University of Pennsylvania, Philadelphia, PA 19104-6395, USA. E-mail: {\tt pop@math.upenn.edu}}


\begin{thebibliography}{1}

\bibitem{abh}
Shreeram S.~Abhyankar, On the valuations centered in a local domain, American Journal of Mathematics, vol.~78 (1956), 321--348.

\bibitem{artin}
Michael Artin, Algebraic approximation of structures over complete local
  rings, Institut des Hautes \'{E}tudes Scientifiques. Publications Math\'{e}matiques, no.~36 (1969), 23--58.

\bibitem{aschDH}
Matthias Aschenbrenner, Lou van~den Dries, and Joris van~der Hoeven,
  \emph{Asymptotic differential algebra and model theory of transseries},
  Annals of Math.\ Studies, vol.~195, Princeton Univ.\ Press,
  Princeton, NJ, 2017.
  
\bibitem{Bou} Nicolas Bourbaki, {\em Commutative Algebra}, Hermann and Addison-Wesley, 1972.

\bibitem{blr}
Siegfried Bosch, Werner L\"utkebohmert, and Michel Raynaud, \emph{N\'eron
  models}, Ergebnisse der Mathematik und ihrer Grenzgebiete (3), vol.~21, Springer-Verlag, Berlin, 1990.
 
\bibitem{Cher}  
Vladimir~I.~Chernousov, The Hasse principle for groups of type E8 (Russian),  Doklady Akademii Nauk SSSR, vol.~360 (1989), no.~5, 1059--1063; translation in 
Soviet Mathematics Doklady, vol.~39 (1989), no. 3,~592--596. 

\bibitem{CTOP}
Jean-Louis Colliot-Th\'{e}l\`ene, Manuel~Ojanguren, and Raman Parimala, Quadratic
 forms over fraction fields of two-dimensional {H}enselian rings and {B}rauer
 groups of related schemes, Algebra, arithmetic and geometry, {P}art {I},
 {II} ({M}umbai, 2000), Tata Institute of Fundamental Research Studies in Mathematics, vol.~16 (2002), 185--217.

\bibitem{CTPS1}
Jean-Louis Colliot-Th\'{e}l\`ene, Raman Parimala, and Venapally Suresh,
 Patching and local-global principles for homogeneous spaces over
  function fields of {$p$}-adic curves, Commentarii Mathematici Helvetici, vol.~87 (2012), no.~4, 1011--1033.

\bibitem{CTPS2}
Jean-Louis Colliot-Th\'{e}l\`ene, Raman Parimala, and Venapally Suresh, Lois de
  r\'{e}ciprocit\'{e} sup\'{e}rieures et points rationnels, Transactions of the American Mathematical Society, vol.~368 (2016), no.~6, 4219--4255.

\bibitem{englerprestel}
Antonio~J. Engler and Alexander Prestel, \emph{Valued fields}, Springer
  Monographs in Mathematics, Springer-Verlag, Berlin, 2005.

\bibitem{GGMB}
Ofer Gabber, Philippe Gille, and Laurent Moret-Bailly,
Fibr\'es principaux sur les corps valu\'es hens\'eliens,
Algebraic Geometry, vol.~5 (2014), 573--612.

\bibitem{greenpoproquette}
Barry~W. Green, Florian Pop, and Peter~J. Roquette, On Rumely's
  local-global principle, Jahresbericht der Deutschen Mathematiker-Vereinigung, vol.~97 (1995), no.~2, 43--74.
  
\bibitem{HH}
David Harbater and Julia Hartmann, Patching over fields, Israel J.~of Math., vol.~176 (2010), 61--107.  
  
 \bibitem{hhk1}
David Harbater, Julia Hartmann, and Daniel Krashen, Applications of patching to quadratic forms and central simple algebras, Inventiones Mathematicae, vol.~178 (2009), no.~2, 231--263.

\bibitem{hhkWeier}
David Harbater, Julia Hartmann, and Daniel Krashen, 
Weierstrass preparation and algebraic invariants,  Mathematische Annalen, vol.~356, no.~4 (2013),
   1405--1424.
  
\bibitem{hhk2}
David Harbater, Julia Hartmann, and Daniel Krashen, Local-global principles for Galois cohomology, Commentarii Mathematici Helvetici, vol.~89 (2014), no.~1, 215--253.
  
\bibitem{hhktorsors}
David Harbater, Julia Hartmann, and Daniel Krashen, Local-global
  principles for torsors over arithmetic curves, American Journal of Mathematics, vol.~137 (2015), no.~6, 1559--1612.
  
\bibitem{lipman}
Joseph Lipman, Rational singularities, with applications to algebraic
  surfaces and unique factorization,  Institut des Hautes \'{E}tudes Scientifiques, Publications Math\'{e}matiques, no.~36 (1969), 195--279.
  
\bibitem{meh}
Vler\"e Mehmeti, Patching over Berkovich curves and quadratic forms, 
Compos.~Math., vol.~155 (2019), no.~12, 2399--2438.

\bibitem{PPS}
R.~Parimala, R.~Preeti, and V.~Suresh, Local-global principle for reduced
  norms over function fields of $p$-adic curves, Compositio Mathematica, vol.~154 (2018), no.~2, 410--458.

\bibitem{popstix}
Florian Pop and Jakob Stix, Arithmetic in the fundamental group of a
  $p$-adic curve. On the $p$-adic section conjecture for curves, Journal f\"{u}r die Reine und Angewandte Mathematik, vol.~725 (2017), 1--40.

\bibitem{sansuc}
Jean-Jacques Sansuc, Groupe de  Brauer et arithm\'{e}tique des groupes
  alg\'{e}briques lin\'{e}aires sur un corps de nombres, Journal f\"{u}r die Reine und Angewandte Mathematik, vol.~327 (1981), 12--80.

\bibitem{schilling}
Otto F.~G.~Schilling, \emph{The Theory of Valuations}, Mathematical Surveys, no. 4, American Mathematical Society, New York, N. Y., 1950.

\bibitem{tigwad}
Jean-Pierre Tignol and Adrian~R. Wadsworth, \emph{Value functions on simple
  algebras, and associated graded rings}, Springer Monographs in Mathematics,
  Springer, Cham, 2015.

\bibitem{warnerfields}
Seth Warner, \emph{Topological fields}, North-Holland Mathematics Studies, vol.~157, Notas de Matem\'atica
  [Mathematical Notes], no.~126, North-Holland Publishing Co., Amsterdam, 1989.

\end{thebibliography}
\end{document}